\documentclass[final,11p,times]{elsarticle}
\usepackage{amssymb}
\usepackage{amsmath}
\usepackage{amsthm}
\usepackage[utf8]{inputenc}
\usepackage{marginnote}
\usepackage{amsfonts}
\usepackage{textcomp}
\usepackage{color,soul}
\usepackage{placeins}
\usepackage{tikz}
\usepackage{makecell}
\usepackage{amssymb}
\usepackage{graphicx}
\usepackage{hyperref}
\usepackage[none]{hyphenat}
\usepackage{lipsum}
\theoremstyle{mythmstyle}
\newtheorem{theorem}{Theorem}[section]
\newtheorem{proposition}[theorem]{Proposition}
\newtheorem{lemma}[theorem]{Lemma}

\theoremstyle{definition}
\newtheorem{definition}[theorem]{Definition}
\newtheorem{remark}[theorem]{Remark}
\newtheorem{example}[theorem]{Example}
\allowdisplaybreaks

\newcommand{\inprod}[2]{\langle#1 \vphantom{#2}, #2\vphantom{#1}\rangle}

\journal{}
\begin{document}
	\begin{frontmatter}
		
		\title{Vanishing Moment Dilations and Applications to Balanced
and Graph Frames}
 \author{Deguang Han}
		\ead{deguang.han@ucf.edu}
\author{Shankhadeep Mondal*}
		\ead{shankhadeep.mondal@ucf.edu}
       \author{Kai Liu}
		\ead{kai.liu@ucf.edu}
		\author{Omar Nour}
		\ead{omar.nour@ucf.edu}

		\address{Department of Mathematics, University of Central Florida, 4000 Central Florida Blvd, Orlando, USA}
		\cortext[shan]{Corresponding Author: Shankhadeep Mondal (shankhadeep.mondal@ucf.edu)}		
		
		\begin{abstract}
		Given a matrix $A\in M_{k\times N}(\Bbb{F}),$ we examine finite frames $\{f_j\}_{j=1}^N$ with vanishing $A$-moments and investigate their connections with the frame dilation property. We establish several equivalent characterizations for the existence of vanishing $A-$moment dilations, including conditions involving dual frames, operator ranges, and linear independence properties associated with the matrix $A$. In particular, we show that every frame with redundancy $k$ admits a dual frame with vanishing $A$-moments for some totally nonsingular matrix $A$.  A key takeaway of this is that by invoking such a dual frame as the encoding frame, any frame with redundancy $k$ (erasure robustness is not needed)  allows perfect signal reconstructions for any number of $\ell$-erasures ($\ell \leq k$) which only involves  inverting a matrix of small size ($\ell$)  instead of inverting the subframe operator which is  a matrix of large size ($N-\ell$). A special case of this type of frames are  $k$-th vanishing moment frames which have natural connections with graph frames.  We characterize the existence of dual frames with $k-th$  vanishing moments through polynomial cancellation identities and operator theoretical conditions. We further discuss some applications to problems in erasure recovery and  graph frames. These results provide a unified framework linking vanishing moment conditions, dual frame theory, dilation theory, and graph theoretic constructions of frames.
		\end{abstract}

		
	\end{frontmatter}

  \section{Introduction}

Redundancy is one of the central features that distinguishes frames
from bases. A signal represented by a frame may admit several
reconstruction formulas, and this nonuniqueness becomes particularly
useful when the reconstruction system is required to satisfy additional
algebraic, geometric, or stability conditions. Since their introduction
by Duffin and Schaeffer \cite{duff}, frames have therefore become an
important tool in harmonic analysis and in applications where stable
and flexible representations are essential.

An important manifestation of this flexibility is the ability of
frames to withstand the loss of coefficients. Early contributions in
this direction established fundamental connections between redundancy,
tightness, and erasure robustness. Casazza and Kova\v{c}evi\'c
\cite{casa2} studied equal-norm tight frames under coefficient losses,
while Goyal, Kova\v{c}evi\'c, and Kelner \cite{goya} showed the
optimality of uniform tight frames for single erasures from a
coding-theoretic perspective. Holmes and Paulsen \cite{holm} formulated
the reconstruction problem through the operator norm of the associated
error operator and obtained corresponding optimality results for
Parseval frames and canonical duals. These developments established
erasure robustness as an important aspect of finite frame theory.

Several approaches to optimal dual-frame reconstruction have been developed using operator norms, spectral and numerical radii, averaged errors, and probabilistic models of coefficient loss
\cite{ara,arab,arab1,shan1,bodm1,deep1,deep,dev,jins,jerr,shan2,shan5,sali,peh}.
In particular, probabilistic erasure models and their associated
weighted reconstruction problems have been studied extensively in
\cite{shan,shan1,leng3,leng,li,li1}.
Motivated by this principle, we study dual frames satisfying prescribed
linear moment constraints. For a matrix
$A=(a_{ij})\in\mathbb{F}^{m\times N}$ and a finite sequence
$F=\{f_j\}_{j=1}^{N}$, we say that $F$ has vanishing $A$-moments when
\[
\sum_{j=1}^{N}a_{ij}f_j=0,
\qquad i=1,\ldots,m.
\]
Thus, a single matrix $A$ simultaneously encodes several linear
cancellation conditions. Our objective is to determine when these
conditions can be realized by an appropriate dual frame or, more
generally, by a dilation of the original frame.

Our first main result characterizes the vanishing $A$-moment dilation
property in several equivalent ways. In particular, we show that it is
equivalent to the transversality condition
\[
\operatorname{Row}(A)\cap\operatorname{Ran}(\Theta_F)=\{0\},
\]
and to the existence of a dual frame satisfying the prescribed
$A$-moment conditions. This connects moment cancellation with the
geometry of the analysis operator and dual-frame theory. We then apply
this viewpoint to erasure robustness and characterize robust frames
through matrices encoding linear dependencies among their vectors.
This brings full-spark and totally nonsingular matrices naturally into
the framework.

We further specialize the $A$-moment formulation to higher-order
vanishing moments, where the associated moment spaces determine the
existence of duals satisfying polynomial cancellation conditions. The
first-order case leads to balanced frames, characterized by $\sum_{i=1}^{N}f_i=0.$
This observation also provides a natural connection with graph frames:
the Laplacian of a finite graph annihilates the all-ones vector. We
exploit this connection to study the relationship between balanced
frames and graph frames and, using linearly connected components,
obtain a componentwise characterization of balanced frames that are
equivalent to graph frames.

The paper is organized as follows. Section~2 recalls the necessary
background and introduces vanishing $A$-moments. In Section~3, we
characterize the vanishing $A$-moment dilation property and study its
connections with dual frames, erasure robustness, full-spark frames,
and totally nonsingular relation matrices. Section~4 treats
higher-order vanishing moments, balanced frames, and their relationship
with graph frames, including a characterization through linearly
connected components. We conclude with several directions for further
investigation.

\section{Preliminaries}
Let $\mathcal{H}_n$ denote an $n$-dimensional (real or complex) Hilbert space. 
A finite sequence of elements $F=\{f_i\}_{i=1}^N \subset \mathcal{H}_n$ is called a \textit{frame} for $\mathcal{H}_n$ if there exist constants $A,B>0$ such that
\[
A\|f\|^2 \leq \sum_{i=1}^N |\langle f,f_i\rangle|^2 \leq B\|f\|^2,
\qquad \forall f\in \mathcal{H}_n .
\]

The numbers $A$ and $B$ are called the \textit{frame bounds}. They are not unique. 
The \textit{optimal lower frame bound} is the supremum over all lower frame bounds, and the \textit{optimal upper frame bound} is the infimum over all upper frame bounds.

A frame $F=\{f_i\}_{i=1}^N$ is called \textit{normalized} if $\|f_i\|=1$ for every $i$. 
If $A=B$, that is,
\[
\sum_{i=1}^N |\langle f,f_i\rangle|^2 = A\|f\|^2, \qquad \forall f\in \mathcal{H}_n,
\]
then $F$ is called a \textit{tight frame}. If $A=B=1$, then $F$ is called a \textit{Parseval frame}. Every finite sequence $\{f_i\}_{i=1}^R$ in $\mathcal{H}_n$ is a frame for the Hilbert space $W := \mathrm{span}\{f_i\}_{i=1}^R .$

Let $F=\{f_k\}_{k=1}^N$ be a frame for $\mathcal{H}_n$. The linear mapping 
\[
\Theta_F : \mathcal{H}_n \rightarrow \mathbb{C}^N, \qquad 
\Theta_F(f)=\{\langle f,f_i\rangle\}_{i=1}^N
\]
is called the \textit{analysis operator}.  The adjoint operator $\Theta_F^* : \mathbb{C}^N \rightarrow \mathcal{H}_n$ defined by
\[
\Theta_F^*\left(\{c_i\}_{i=1}^N\right)=\sum_{i=1}^N c_i f_i
\]
is called the \textit{synthesis operator} (or \textit{preframe operator}). The \textit{frame operator} $S_F : \mathcal{H}_n \to \mathcal{H}_n$ is defined by
\[
S_F f = \Theta_F^*\Theta_F f = \sum_{i=1}^N \langle f,f_i\rangle f_i .
\]

The operator $S_F$ is positive and invertible on $\mathcal{H}_n$, which yields the reconstruction formula
\[
f=\sum_{i=1}^N \langle f,S_F^{-1}f_i\rangle f_i = \sum_{i=1}^N \langle f,f_i\rangle S_F^{-1}f_i.
\]

A frame $G=\{g_i\}_{i=1}^N$ in $\mathcal{H}_n$ is called a \textit{dual frame} of $F=\{f_i\}_{i=1}^N$ if every $f\in \mathcal{H}_n$ admits the representations
\[
f = \sum_{i=1}^N \langle f,f_i\rangle g_i
  = \sum_{i=1}^N \langle f,g_i\rangle f_i .
\]

It is known that $S_F^{-1}F=\{S_F^{-1}f_i\}_{i=1}^N$ is a frame for $\mathcal{H}_n$, called the \textit{ standard dual or canonical dual } of $F$. Moreover for $N>n$, there exist infinitely many dual frames of $F$ \cite{ole}. Every dual frame 
$G=\{g_i\}_{i=1}^N$ of $F$ can be written in the form
\[
g_i = S_F^{-1}f_i + u_i , \qquad i=1,\dots,N,
\]
where the sequence $\{u_i\}_{i=1}^N$ satisfies
\[
\sum_{i=1}^N \langle f,f_i\rangle u_i = 
\sum_{i=1}^N \langle f,u_i\rangle f_i =0,
\qquad \forall f\in \mathcal{H}_n .
\]

\begin{definition}
Let $A \in \mathbb{C}^{k \times N}$. A sequence $\{g_i\}_{i=1}^N$ is said to have \emph{vanishing $A$-moment} if 
$$A \left[ \langle f, g_1 \rangle, \ldots, \langle f, g_N \rangle  \right]^t = 0, \quad  \forall f \in \mathcal{H}_n.$$
\end{definition}

\begin{definition}
A frame $F$ is said to have the \emph{vanishing $A$-moment dilation property} if there exists a sequence $H = \{h_i\}_{i=1}^N$ such that
$$\Theta_F^* \Theta_H = 0\quad \text{and}\quad A\left[ f_1 + h_1, \ldots, f_N + h_N \right]^t =0 .$$
\end{definition}

Let $F=\{f_i\}_{i=1}^{N}$ be a frame for an $n$-dimensional
Hilbert space $\mathcal H_n$. The frame $F$ is said to be
\emph{full spark} if every subset of $F$ consisting of $n$
vectors is linearly independent. Equivalently, every $n$-element
subset of $F$ is a basis for $\mathcal H_n$.

\subsection{Exact Recovery}

We first consider the situation in which the matrix defining the
vanishing moment conditions possesses sufficient nonsingularity to
recover the erased coefficients exactly. The main advantage of this
approach is that the recovery requires only the solution of a linear
system whose size is determined by the number of erasures, rather than
the inversion of the frame operator associated with the surviving
frame vectors.

Let $F=\{f_j\}_{j=1}^{N}$ be an $A$-balanced frame for
$\mathcal H_n$, where $A=(a_{ij})\in\mathbb F^{m\times N}$. Thus,
\[
\sum_{j=1}^{N}a_{ij}f_j=0,
\qquad i=1,\ldots,m.
\]
For $x\in\mathcal H_n$, put
\[
c_j=\langle x,f_j\rangle,
\qquad
c=(c_1,\ldots,c_N)^t.
\]
The $A$-balanced condition yields
\[
Ac=0.
\]

Suppose that the coefficients indexed by
$\Lambda\subset\{1,\ldots,N\}$ are erased. Writing
$c_\Lambda=(c_j)_{j\in\Lambda}$, we obtain
\[
A_\Lambda c_\Lambda
=
-A_{\Lambda^c}c_{\Lambda^c},
\]
where the right-hand side is completely determined by the available
coefficients. Consequently, whenever $A_\Lambda$ has full column rank,
the erased coefficients are uniquely determined by
\[
c_\Lambda
=
-\left(A_\Lambda^*A_\Lambda\right)^{-1}
 A_\Lambda^*A_{\Lambda^c}c_{\Lambda^c}.
\]
In particular, if $|\Lambda|=m$ and $A_\Lambda$ is invertible, then
\[
c_\Lambda
=
-A_\Lambda^{-1}A_{\Lambda^c}c_{\Lambda^c}.
\]

Thus, if $A$ is totally nonsingular, every collection of at most $m$
erased coefficients can be recovered exactly. For a single erasure $\Lambda=\{\lambda\}$, one may use any row
$i$ for which $a_{i\lambda}\neq0$ to obtain
\[
c_\lambda
=
-\frac{1}{a_{i\lambda}}
\sum_{j\neq\lambda}a_{ij}c_j.
\]
Hence, if $A$ is totally nonsingular, every single erased coefficient
is recoverable from the remaining coefficients.

For two erasures $\Lambda=\{\lambda_1,\lambda_2\}$, choose two rows
$i_1,i_2$ such that
\[
\begin{pmatrix}
a_{i_1\lambda_1} & a_{i_1\lambda_2}\\
a_{i_2\lambda_1} & a_{i_2\lambda_2}
\end{pmatrix}
\]
is invertible. Then
\[
\begin{pmatrix}
c_{\lambda_1}\\
c_{\lambda_2}
\end{pmatrix}
=
-
\begin{pmatrix}
a_{i_1\lambda_1} & a_{i_1\lambda_2}\\
a_{i_2\lambda_1} & a_{i_2\lambda_2}
\end{pmatrix}^{-1}
\begin{pmatrix}
\displaystyle\sum_{j\notin\Lambda}a_{i_1j}c_j\\[2mm]
\displaystyle\sum_{j\notin\Lambda}a_{i_2j}c_j
\end{pmatrix}.
\]
In particular, total nonsingularity of $A$ guarantees exact recovery
from every pattern of one or two erasures.

A particularly useful deterministic choice of the relation matrix is
provided by a Vandermonde matrix. Let
$t_1,\ldots,t_N$ be distinct positive numbers and consider
\[
A=
\begin{pmatrix}
1&1&\cdots&1\\
t_1&t_2&\cdots&t_N\\
t_1^2&t_2^2&\cdots&t_N^2\\
\vdots&\vdots&&\vdots\\
t_1^{m-1}&t_2^{m-1}&\cdots&t_N^{m-1}
\end{pmatrix}.
\]
For a suitable ordering $0<t_1<\cdots<t_N$, this matrix is strictly
totally positive, and hence, in particular, totally nonsingular.
Thus every square submatrix of $A$ of order at most $m$ is
invertible.

Suppose that $F=\{f_j\}_{j=1}^{N}$ is $A$-balanced and that
$\Lambda=\{\lambda_1,\ldots,\lambda_r\}$, with $r\leq m$, is the set
of erased coefficients. Using the first $r$ moment equations, we obtain
\[
V_\Lambda
\begin{pmatrix}
c_{\lambda_1}\\
\vdots\\
c_{\lambda_r}
\end{pmatrix}
=
-
\begin{pmatrix}
\displaystyle\sum_{j\notin\Lambda}c_j\\[2mm]
\displaystyle\sum_{j\notin\Lambda}t_jc_j\\
\vdots\\
\displaystyle\sum_{j\notin\Lambda}t_j^{r-1}c_j
\end{pmatrix},
\]
where
\[
V_\Lambda=
\begin{pmatrix}
1&1&\cdots&1\\
t_{\lambda_1}&t_{\lambda_2}&\cdots&t_{\lambda_r}\\
\vdots&\vdots&&\vdots\\
t_{\lambda_1}^{r-1}&t_{\lambda_2}^{r-1}
&\cdots&t_{\lambda_r}^{r-1}
\end{pmatrix}.
\]
Since the nodes $t_{\lambda_1},\ldots,t_{\lambda_r}$ are distinct,
\[
\det V_\Lambda
=
\prod_{1\leq p<q\leq r}
(t_{\lambda_q}-t_{\lambda_p})
\neq0.
\]
Consequently,
\[
\begin{pmatrix}
c_{\lambda_1}\\
\vdots\\
c_{\lambda_r}
\end{pmatrix}
=
-
V_\Lambda^{-1}
\begin{pmatrix}
\displaystyle\sum_{j\notin\Lambda}c_j\\[2mm]
\displaystyle\sum_{j\notin\Lambda}t_jc_j\\
\vdots\\
\displaystyle\sum_{j\notin\Lambda}t_j^{r-1}c_j
\end{pmatrix}.
\]

Moreover, the inverse $V_\Lambda^{-1}$ can be written explicitly in
terms of the Lagrange interpolation polynomials. For
$q=1,\ldots,r$, define
\[
L_q(t)
=
\prod_{\substack{\ell=1\\ \ell\neq q}}^{r}
\frac{t-t_{\lambda_\ell}}
     {t_{\lambda_q}-t_{\lambda_\ell}}
=
\sum_{k=0}^{r-1}\gamma_{qk}t^k.
\]
Then
\[
V_\Lambda^{-1}
=
(\gamma_{q,k-1})_{q,k=1}^{r}.
\]
Hence the missing frame coefficients can be recovered explicitly from
the surviving coefficients without reconstructing the signal first.
This gives a deterministic and constructive recovery procedure for
every erasure pattern of cardinality at most $m$.

For example, in the case of a single erasure
$\Lambda=\{\lambda\}$, the zeroth moment relation immediately gives
\[
c_\lambda=-\sum_{j\neq\lambda}c_j.
\]

For two erasures
$\Lambda=\{\lambda_1,\lambda_2\}$, the first two moment relations give
\[
\begin{pmatrix}
1&1\\
t_{\lambda_1}&t_{\lambda_2}
\end{pmatrix}
\begin{pmatrix}
c_{\lambda_1}\\
c_{\lambda_2}
\end{pmatrix}
=
-
\begin{pmatrix}
\displaystyle\sum_{j\notin\Lambda}c_j\\[2mm]
\displaystyle\sum_{j\notin\Lambda}t_jc_j
\end{pmatrix}.
\]
Since $t_{\lambda_1}\neq t_{\lambda_2}$,
\[
\begin{pmatrix}
1&1\\
t_{\lambda_1}&t_{\lambda_2}
\end{pmatrix}^{-1}
=
\frac{1}{t_{\lambda_2}-t_{\lambda_1}}
\begin{pmatrix}
t_{\lambda_2}&-1\\
-t_{\lambda_1}&1
\end{pmatrix}.
\]
Therefore,
\[
c_{\lambda_1}
=
\frac{
\displaystyle
\sum_{j\notin\Lambda}(t_j-t_{\lambda_2})c_j
}{
t_{\lambda_2}-t_{\lambda_1}
}\quad \text{and} \quad c_{\lambda_2}
=
\frac{
\displaystyle
\sum_{j\notin\Lambda}(t_{\lambda_1}-t_j)c_j
}{
t_{\lambda_2}-t_{\lambda_1}}.
\]

Thus, once the erasure locations are known, one or two missing
coefficients can be recovered directly from the surviving
coefficients by explicit formulas.

\begin{example}
Let $\mathcal H_2=\mathbb R^2$ and consider the frame
$F=\{f_j\}_{j=1}^{5}$ given by
\[
f_1=
\begin{pmatrix}1\\0\end{pmatrix},
\quad
f_2=
\begin{pmatrix}-3\\1\end{pmatrix},
\quad
f_3=
\begin{pmatrix}3\\-3\end{pmatrix},
\quad
f_4=
\begin{pmatrix}-1\\3\end{pmatrix},
\quad
f_5=
\begin{pmatrix}0\\-1\end{pmatrix}.
\]
Consider the Vandermonde-type matrix
\[
A=
\begin{pmatrix}
1&1&1&1&1\\
1&2&3&4&5\\
1&4&9&16&25
\end{pmatrix}.
\]
Since the nodes $1,2,3,4,5$ are distinct and nonzero, every square
submatrix of $A$ of order at most $3$ is nonsingular. Hence $A$ is
totally nonsingular. A direct verification gives $\sum_{j=1}^{5}f_j=0,\quad
\sum_{j=1}^{5}j f_j=0,\quad
\sum_{j=1}^{5}j^2 f_j=0.$
Thus $F$ is $A$-balanced.

For $x\in\mathbb R^2$, let $c_j=\langle x,f_j\rangle,\quad j=1,\ldots,5.$ Then $Ac=0$. Suppose, for instance, that the coefficients
$c_1,c_3,c_5$ are erased, while $c_2$ and $c_4$ are available.
The three moment relations give
\[
\begin{pmatrix}
1&1&1\\
1&3&5\\
1&9&25
\end{pmatrix}
\begin{pmatrix}
c_1\\c_3\\c_5
\end{pmatrix}
=
-
\begin{pmatrix}
1&1\\
2&4\\
4&16
\end{pmatrix}
\begin{pmatrix}
c_2\\c_4
\end{pmatrix}.
\]
The $3\times3$ coefficient matrix is nonsingular, and therefore the
three missing coefficients are uniquely determined. More explicitly,
\[
c_1=-\frac{3}{8}c_2+\frac{1}{8}c_4,\qquad
c_3=-\frac{3}{4}c_2-\frac{3}{4}c_4,\qquad
c_5=\frac{1}{8}c_2-\frac{3}{8}c_4.
\]

Hence three erased frame coefficients can be recovered exactly from
the two surviving coefficients. Moreover, since $A$ is totally
nonsingular, the argument is independent of the particular erasure
pattern: \emph{every collection of at most three erased coefficients
can be recovered exactly from the remaining coefficients}. This
illustrates how vanishing $A$-moments provide a uniform algebraic
recovery mechanism for multiple erasures.
\end{example}

\subsection{Approximation}
As computing matrix inverses for exact recovery may be computationally expensive in practice, we implement the method of compressive sensing by considering when the matrix defining the vanishing moment condition possesses the restricted isometry property.  Let $A\in \mathbb{C}^{m\times N}$ and let $k\in\mathbb{N}$.  We say that $A$ satisfies the \textit{Restricted Isometry Property (RIP) of order $k$} if there exists a constant $\delta_k\in (0,1)$ such that for every $k$-sparse vector $x\in\mathbb{C}^N$ (i.e., having at most $k$ non-zero entries) the following inequality holds:
\[
(1-\delta_k)\|x\|_2^2
\le
\|A x\|_2^2
\le
(1+\delta_k)\|x\|_2^2.
\]
The smallest such constant $\delta_k$ is called the \textit{Restricted Isometry Constant (RIC)} of order $k$. Note this is equivalent to $\|I_k-A_k^*A_k\|\leq \delta_k$ for every $m\times k$ submatrix $A_k$

Let $F=\{f_i\}_{i=1}^N$ be an $A$-balanced frame, where $A$ is a $k\times N$ matrix with the $k$-RIP, and let $G=\{g_i\}_{i=1}^N$ be a dual frame for $F$. Suppose that $\Lambda=\{\lambda_1,\dots,\lambda_k\}\subset\{1,\dots,N\}$ is an index set corresponding to  $k$-erasures. The erasure operator $\mathcal{E}_\Lambda$ is defined by 
$$\mathcal{E}_\Lambda\coloneqq\Theta_G^*D_\Lambda\Theta_F=\sum_{i\in\Lambda}\inprod{x}{f_i}g_i,$$
where $D_\Lambda$ is an $N\times N$ diagonal matrix with $d_{ii}=1$ if $i\in\Lambda$ and zero otherwise. Define $\mathcal{R}_\lambda\coloneqq I-\mathcal{E}_\Lambda=\sum_{i\in\Lambda^c}\inprod{x}{f_i}g_i.$ By the reconstruction formula, any vector $x$ can be written as 
\begin{equation}
x=\sum_{i\in\Lambda}\inprod{x}{f_i}g_i+\sum_{i\in\Lambda^c}\inprod{x}{f_i}g_i=\mathcal{E}_\Lambda x+\mathcal{R}_\Lambda x,
\end{equation}
where $\Lambda^c=\{1,\dots,N\}\backslash\Lambda$.

Now suppose that $F$ is $A$-balanced, where $A$ is a $k\times N$ matrix with the $k$-RIP, then the erased component $\mathcal{E}_\Lambda x$ of the signal $x$ may be approximated by first observing that the erased coefficients can be expressed in terms of the non-erased coefficients via the $A$-balanced property:
$$\sum_{j\in\Lambda}a_{ij}\inprod{x}{f_j}=-\sum_{j\in\Lambda^c}a_{ij}\inprod{x}{f_j}, \quad \text{for }i=1,\dots,k.$$
Next, observe that the $k$-RIP of $A$ is equivalent to $\|D_\Lambda-D_\Lambda A^*AD_\Lambda\|\le \delta_k$ for every $D_\Lambda$ with $|\Lambda|\leq k$. Denoting $\Tilde{E}^{(1)}=D_\Lambda A^*AD_\Lambda$, the first order approximation of $x$ is given by 

$$x^{(1)}=\Theta_G^*\Tilde{E}^{(1)}_\Lambda\Theta_Fx+\mathcal{R}_\Lambda x.$$

Higher order approximations can be iteratively produced for better accuracy by employing the following iteration:
\begin{align*}
x^{(2)}&=\Theta_G^*\Tilde{E}^{(2)}_\Lambda\Theta_Fx+\mathcal{R}_\Lambda x,\\
&\;\;\vdots\\
x^{(m)}&=\Theta_G^*\Tilde{E}^{(m)}_\Lambda\Theta_Fx+\mathcal{R}_\Lambda x,
\end{align*}
where $\Tilde{E}^{(m)}_\Lambda=\Tilde{E}^{(1)}_\Lambda+(D_\Lambda-\Tilde{E}^{(1)}_\Lambda)\Tilde{E}^{(m-1)}_\Lambda,$ for $m\ge2.$ The error of this approximation is given by
\begin{align*}
    x-x^{(m)}&= \Theta_G^*\left[ D_\Lambda-\Tilde{E}^{(m)}_\Lambda \right]\Theta_Fx\\
    &=\Theta_G^*\left[ D_\Lambda-\Tilde{E}^{(1)}_\Lambda-(D_\Lambda-\Tilde{E}^{(1)}_\Lambda)\Tilde{E}^{(m-1)}_\Lambda \right]\Theta_Fx\\
    &=\Theta_G^*\left[ \left(D_\Lambda-\Tilde{E}^{(1)}_\Lambda\right)\left(D_\Lambda-\Tilde{E}^{(m-1)}_\Lambda\right) \right]\Theta_Fx\\
    &=\Theta_G^*\left[ \left(D_\Lambda-\Tilde{E}^{(1)}_\Lambda\right)^2\left(D_\Lambda-\Tilde{E}^{(m-2)}_\Lambda\right) \right]\Theta_Fx\\
    &\;\;\vdots\\
    &=\Theta_G^*\left(D_\Lambda-\Tilde{E}^{(1)}_\Lambda\right)^m\Theta_Fx\\
    &=\Theta_G^*\left(D_\Lambda-D_\Lambda A^*AD_\Lambda\right)^m\Theta_Fx.
\end{align*}

Thus, we have
\begin{align*}
    \|x-x^{(m)}\|&=\left\| \Theta_G^*\left(D_\Lambda-\Tilde{E}^{(1)}_\Lambda\right)^m\Theta_Fx \right\|\\
    &=\left\| \Theta_G^*\left(D_\Lambda-D_\Lambda A^*AD_\Lambda\right)^m\Theta_Fx \right\|\\
    &\le \| \Theta_G^*\|\,\|D_\Lambda-D_\Lambda A^*AD_\Lambda\|^m\,\|\Theta_F\| \|x\| \\
    &\leq \delta_k^m \sqrt{B_F} \sqrt{B_G}\|x\|,
\end{align*}
where $B_F$, $B_G$ their respective upper frame bounds.

\section{Vanishing $A$-Moment Dilations}

We now turn to the problem of characterizing frames that admit vanishing $A$-moment dilations. The main result of this section establishes several equivalent conditions involving dual frames, operator ranges, and linear independence properties associated with the matrix $A$. These equivalences reveal the underlying geometric structure governing the existence of vanishing $A$-moment dilations. 

Given two finite sequences of vectors $X=\{x_i\}_{i=1}^N
\quad\text{and}\quad Y=\{y_i\}_{i=1}^N,$ we write \(X \preceq Y\) if every linear dependence relation among the
vectors of \(Y\) is also a linear dependence relation among the vectors
of \(X\); that is,
\[
\sum_{i=1}^N c_i y_i=0
\quad\Longrightarrow\quad
\sum_{i=1}^N c_i x_i=0
\]
for every choice of scalars \(c_1,\ldots,c_N\).

Let $F = \{f_j\}_{j=1}^N$ be a finite sequence in a Hilbert space $\mathcal{H}_n$ and let $A = (a_{ij}) \in M_{k\times N}(\mathbb{C}).$ Let $a_1,\dots,a_k \in \mathbb{C}^N$ denote the row vectors of $A$, and let $\Theta_F : \mathcal{H}_n \to \mathbb{C}^N$ be the analysis operator of $F$. Let $P$ be the orthogonal projection of $\mathbb{C}^N$ onto $\operatorname{Ran}(\Theta_F)$. Define
\[
x_i := \sum_{j=1}^N a_{ij} f_j = \Theta_F^* a_i, 
\quad
y_i := P^\perp a_i, 
\quad i=1,\dots,k.
\]
Let $X = \{x_i\}_{i=1}^k$ and $Y = \{y_i\}_{i=1}^k$.

\begin{theorem}\label{equivalancethm3point1}
Let $F=\{f_i\}_{i=1}^N$ be a frame for $\mathcal{H}_n$ and let $A \in \mathbb{C}^{k\times N}$ have full row rank, with row vectors $\{a_i\}_{i=1}^k$. Then the following are equivalent:
\begin{enumerate}
    \item $F$ has the vanishing $A$-moment dilation property.
    
    \item $F$ has a dual frame $G=\{g_i\}_{i=1}^N$ such that $G$ has vanishing $A$-moment.
    
    \item $\mathrm{Row}(A) \cap \mathrm{Ran}(\Theta_F) = \{0\}$.
    
    \item The family $\{P^\perp a_i\}_{i=1}^k$ is linearly independent, where $P$ is the orthogonal projection onto $\mathrm{Ran}(\Theta_F)$.
    
    \item 
     $X \preceq Y$. 
    
\end{enumerate}
\end{theorem}

\begin{proof}

\textbf{(1) $\Leftrightarrow$ (2):} Assume $F$ has the vanishing $A$-moment dilation property. Then there exists a sequence $H=\{h_i\}_{i=1}^N$ such that $\Theta_F^* \Theta_H = 0
\quad \text{and} \quad
A[f_1+h_1, \dots, f_N+h_N]^T = 0.$ Now define
\[
g_i = S_F^{-1}(f_i + h_i), \quad 1 \leq i \leq N.
\]
Then $\{g_i\}_{i=1}^N$ is a dual of $F$. Moreover,
\[
A[g_1,\dots,g_N]^T 
= A[S_F^{-1}(f_1+h_1),\dots,S_F^{-1}(f_N+h_N)]^T = 0.
\]
Thus $G$ has vanishing $A$-moment.\\
Conversely, assume $F$ has a dual $G=\{g_i\}_{i=1}^N$ with vanishing $A$-moment. Let $g_i = S_F^{-1} f_i + h_i,$ where $\Theta_F^* \Theta_H = 0$. Then $A[g_1,\dots,g_N]^T = 0$ implies that $A[f_1 + S_F h_1,\dots,f_N + S_F h_N]^T = 0.$ Therefore, $F$ has vanishing $A$-moment dilation property.

\textbf{(2) $\Leftrightarrow$ (3):} Assume $F$ has a dual $G$ with vanishing $A$-moment. Then $\Theta_G^* a_i = \sum_{j=1}^N a_{ij} g_j = 0.$ If $w \in \mathrm{Row}(A) \cap \mathrm{Ran}(\Theta_F)$, then
$w = \Theta_F f = \sum_{i=1}^k c_i a_i,$ for some scaler $c_i.$ This gives $f = \Theta_G^* \Theta_F f = \sum_{i=1}^k c_i \Theta_G^* a_i = 0,$ and hence  $w=0$. \\
Conversely, assume $\mathrm{Row}(A) \cap \mathrm{Ran}(\Theta_F) = \{0\}$. Define a linear map
\[
L: \mathrm{Ran}(\Theta_F) \oplus \mathrm{Row}(A) \to \mathcal{H}
\]
by
\[
L(\Theta_F f + w) = f.
\]
This is well-defined due to the trivial intersection. Extend $L$ linearly to $\mathbb{C}^N$ and take $g_i = L e_i.$ Then $\{g_i\}_{i=1}^N$ is a dual of $F$ and
\[
\sum_{j=1}^N a_{ij} g_j = L\left(\sum_{j=1}^N a_{ij} e_j\right) = L(a_i) = 0.
\]
Thus $G$ has vanishing $A$-moment.

\textbf{(3) $\Leftrightarrow$ (4):} Note that $\ker\big(P^\perp|_{\mathrm{Row}(A)}\big) = \mathrm{Row}(A) \cap \mathrm{Ran}(\Theta_F).$ Hence the intersection is trivial if and only if $\{P^\perp a_i\}_{i=1}^k$ is linearly independent.

\textbf{(4) $\Leftrightarrow$ (5):} Suppose $\sum_{i=1}^k c_i y_i = 0.$ Then $P^\perp \left(\sum_{i=1}^k c_i a_i \right)=0.$ This gives 
$\sum_{i=1}^k c_i a_i \in \mathrm{Row}(A) \cap \mathrm{Ran}(\Theta_F).$ Thus $\sum c_i a_i=0$ and hence $\sum_{i=1}^k c_i x_i = \sum_{i=1}^k c_i \sum_{j=1}^N a_{ij} f_j = 0.$ So $X \preceq Y$.\\
Conversely, if $X \preceq Y$, define a linear map
\[
L:\mathrm{span}\{y_i\} \to \mathrm{span}\{x_i\}, \quad y_i \mapsto x_i.
\]
Extend $L$ linearly and define
\[
\gamma := -L P^\perp : \mathbb{C}^N \to \mathcal{H}_n.
\]
Then $\gamma = \Theta_H^*$ for some sequence $H=\{h_i\}$. It is easy to verify that $\Theta_H^* \Theta_F = 0$ and $A[f_1 + h_1, \dots, f_N + h_N]^T = 0.$ Thus $F$ has vanishing $A$-moment dilation property.

\end{proof}

\begin{example}
Let
\[
F=\left\{
\begin{pmatrix}1\\0\end{pmatrix},
\begin{pmatrix}0\\1\end{pmatrix},
\begin{pmatrix}1\\1\end{pmatrix},
\begin{pmatrix}1\\-1\end{pmatrix}
\right\}
\]
be a frame for $\mathbb R^2$, and let
\[
A=
\begin{pmatrix}
1&0&0&0\\
0&1&0&0
\end{pmatrix}.
\]
The analysis operator of $F$ is
\[
\Theta_F=
\begin{pmatrix}
1&0\\
0&1\\
1&1\\
1&-1
\end{pmatrix}.
\]
It follows that $\operatorname{Row}(A)\cap\operatorname{Ran}(\Theta_F)=\{0\}.$ Hence, by Theorem~\ref{equivalancethm3point1}, $F$ admits a dual frame with vanishing
$A$-moment and also possesses a vanishing $A$-moment dilation.

Indeed, one such dual frame is
\[
G=\left\{
0,\,0,\,
\frac12\begin{pmatrix}1\\1\end{pmatrix},
\frac12\begin{pmatrix}1\\-1\end{pmatrix}
\right\}.
\]
Clearly,
\[
A[g_1,g_2,g_3,g_4]^T=0.
\]
On the other hand,
\[
A[f_1,f_2,f_3,f_4]^T\neq0,
\]
so the original frame itself does not have vanishing $A$-moment.
Thus, this example illustrates that the vanishing $A$-moment
dilation property may hold even when the original frame does not
have vanishing $A$-moment.
\end{example}

The following proposition establishes a direct connection between the robustness of a frame against erasures and the algebraic structure of the matrix encoding its linear dependencies. In particular, when the number of erased coefficients equals the excess of the frame, robustness is characterized by the nonsingularity of the corresponding square submatrices of the matrix $A.$

\begin{proposition}\label{existanceprop3.3}
Let $F=\{f_i\}_{i=1}^{N}$ be a frame for an $n$-dimensional Hilbert space
$\mathcal H_n$, $N>n$ and 
\[
\Phi=
\begin{bmatrix}
f_{1}^t
\ldots
f_{N}^t
\end{bmatrix}^t.
\]
 Let $r=N-n.$ Then there exists a matrix $A\in\mathbb C^{r\times N}$ such that $A\Phi=0.$ Moreover, every $r\times r$  sub-matrix of $A$ is invertible if and only if
$F$ is $r$-erasure robust.
\end{proposition}

\begin{proof}
Since $F$ is a frame for $\mathcal H_n$, \,$\operatorname{rank}(\Phi)=n.$ Hence $\dim\ker(\Phi^*)=N-n=r.$ Choose a basis $\{a_1^*,\ldots,a_r^*\}$ of $\ker(\Phi^*)$, and define
\[
A=
\begin{bmatrix}
a_{1}^t\\
\vdots\\
a_{r}^t
\end{bmatrix}.
\]
Since $\Phi^*a_j^*=0,\quad j=1,\ldots,r,$ we obtain $A\Phi=0.$ Now let
\[
J\subset\{1,\ldots,N\},
\qquad |J|=r,
\]
and denote by $A_J$ the $r\times r$ submatrix obtained by keeping the columns indexed by $J$.

We first prove that
\[
\det(A_J)=0
\iff
\operatorname{span}\{f_i:i\in J^c\}\neq\mathcal H_n.
\]

Suppose first that $\det(A_J)=0.$ Then there exists a nonzero vector $u\in\mathbb C^{r}$ such that $A_Ju=0.$ Extend $u$ to a vector $c=(c_i)_{i=1}^{N}\in\mathbb C^{N},$ by setting
\[
c_i=
\begin{cases}
u_i,& i\in J,\\
0,& i\in J^c.
\end{cases}
\]

Then $Ac=0.$ Since the rows of $A$ form a basis of $\ker(\Phi^*)$, we have
\[
\ker(A)=\bigl(\ker(\Phi^*)\bigr)^\perp
=\operatorname{Ran}(\Phi).
\]
Hence $c=\Phi x,$ for some nonzero vector $x\in\mathcal H_n$.
Therefore, $c_i=\langle x,f_i\rangle.$ Since $c_i=0$ for every $i\in J^c$, it follows that $\langle x,f_i\rangle=0,
\quad i\in J^c.$ Thus $x$ is orthogonal to every vector in $\{f_i:i\in J^c\},$ which implies $\operatorname{span}\{f_i:i\in J^c\}
\neq
\mathcal H_n.$

Conversely, suppose $\operatorname{span}\{f_i:i\in J^c\}
\neq
\mathcal H_n.$ Then there exists a nonzero vector $x\in\mathcal H_n$ such that $\langle x,f_i\rangle=0,
\quad i\in J^c.$ Set $c=\Phi x.$ Since $\Phi$ has full column rank,  $c\neq0.$ Moreover, $c_i=0,
\quad i\in J^c.$ Also, $c\in\operatorname{Ran}(\Phi)=\ker(A),$ so $Ac=0.$ Since $c$ is supported on $J$, the columns of $A_J$ are linearly dependent. Hence $\det(A_J)=0.$ Since $|J|=r=N-n,$ we have $|J^c|=n.$ Thus every $r\times r$ column submatrix of $A$ is invertible if and only if every $n$-element subset of $F$ spans $\mathcal H_n$, which is equivalent to $F$ being $r$-erasure robust.
\end{proof}

For instance, over $\mathbb R$ one may systematically choose distinct
integer values for $t_1,\ldots,t_r$. Since every square minor of
$V(t_1,\ldots,t_r)A_0$ is a nonzero polynomial in the parameters
$t_1,\ldots,t_r$, only a proper algebraic set of parameter choices
can cause one of these minors to vanish. Consequently, suitable
distinct choices of $t_1,\ldots,t_r$ exist, and the above procedure
produces a totally nonsingular relation matrix.

\begin{lemma}
Let $X,Y,Z$ be vector spaces, and let $S:X\to Y,\quad T:X\to Z$ be linear maps. Then there exists a linear map $M:Y\to Z$ such that $T=MS$
if and only if $\ker(S)\subseteq \ker(T).$
\end{lemma}

The following theorem shows that every finite frame admits a dual frame with the strongest possible robustness permitted by its redundancy. More precisely, if $N=n+k,$ then one can always choose a dual frame that remains a frame after the erasure of any $k$ elements.

\begin{theorem}\label{robustthm3point5}
Let $F=\{f_i\}_{i=1}^{N}$ be a frame for an $n$-dimensional Hilbert space $\mathcal H_n$, where $N=n+k$. Then $F$ has a dual frame $G=\{g_i\}_{i=1}^{N}$ which is $k$-erasure robust. 
\end{theorem}

In order to prove the above Theorem, we make use of the following lemmas.

\begin{lemma}\label{lem3point5}
Let $V$ be a finite-dimensional affine space over $\mathbb{R}$ or $\mathbb{C}$, and let $Z_1,\dots,Z_m \subsetneq V$
be proper algebraic subsets. Then $Z_1\cup Z_2\cup\cdots\cup Z_m \neq V.$
\end{lemma}

\begin{proof}
For each $i$, there exists a nonzero polynomial $p_i$ on $V$ such that $Z_i \subseteq \{x\in V : p_i(x)=0\}.$ Let $p=p_1p_2\cdots p_m.$ Since $p$ is a nonzero polynomial, it cannot vanish identically on $V$. Hence there exists
$x\in V$ with $p(x)\neq0$, which implies $p_i(x)\neq0 \quad \text{for all } i.$ Therefore, $x\notin Z_i \qquad \text{for all } i,$ and consequently $x\notin \bigcup_{i=1}^m Z_i.$ Thus,
$\bigcup_{i=1}^m Z_i \neq V.$
\end{proof}
\begin{proof}[\textbf{Proof of Theorem \ref{robustthm3point5}}]

We first show the following claim.

\medskip

\noindent
\textit{Claim:}
For every subset $I\subset\{1,\ldots,N\}$ with $|I|=n$, there exists a dual frame
$G=\{g_i\}_{i=1}^{N}$ of $F$ such that $\{g_i\}_{i\in I}$ is a basis for
$\mathcal H_n$.

\medskip

Fix such a subset $I$ and write $J=I^c$. Let $F_I=\{f_i\}_{i\in I}, \quad F_J=\{f_i\}_{i\in J}.$ Let
\[
W_J=\operatorname{span}\{f_j:j\in J\}.
\]
We want to construct a dual $G$ whose vectors indexed by $I$ form a basis.

Let $U=W_J^\perp.$ If $x\in U$ and $\langle x,f_i\rangle=0,
\quad \text{for all } i\in I,$ then also $\langle x,f_j\rangle=0, \quad \text{for all } j\in J,$ because $x\in W_J^\perp$. Hence $x$ is orthogonal to every vector of the frame $F$. Since $F$ spans $\mathcal H_n$, this implies $x=0$.
Therefore the restricted analysis map
\[
\Theta_{F_I}:U\to \mathbb C^n
\]
is injective. Define a linear map on the subspace $\Theta_{F_I}(U)\subseteq\mathbb C^n$ by
\[
L(\Theta_{F_I}x)=x,
\qquad x\in U.
\]
This is well-defined and injective. Since $\dim\mathcal H_n=n$, we may extend
$L$ to an invertible linear map
\[
L:\mathbb C^n\to\mathcal H_n.
\]
Let the columns of $L$ be $\{g_i\}_{i\in I}$. Since $L$ is invertible, the family $\{g_i\}_{i\in I}$ is a basis for $\mathcal H_n$. Now define
\[
T=I_{\mathcal H_n}-L\Theta_{F_I}.
\]
For every $x\in U=W_J^\perp$, by construction, $L\Theta_{F_I}x=x.$ Hence $Tx=0 \quad \text{for all }x\in W_J^\perp.$ Thus $T$ depends only on the component of a vector in $W_J$. Since
$\Theta_{F_J}$ is injective on $W_J$, there exists a linear operator
\[
M:\mathbb C^{|J|}\to \mathcal H_n
\]
such that
\[
M\Theta_{F_J}=T.
\]
Let the columns of $M$ be $\{g_j\}_{j\in J}$. Then for every $x\in\mathcal H_n$,
\[
\sum_{i\in I}\langle x,f_i\rangle g_i
+
\sum_{j\in J}\langle x,f_j\rangle g_j
=
L\Theta_{F_I}x+M\Theta_{F_J}x =x,
\]
Since $M\Theta_{F_J}=T=I-L\Theta_{F_I}.$  Therefore $G=\{g_i\}_{i=1}^{N}$ is a dual frame of $F$, and by construction
$\{g_i\}_{i\in I}$ is a basis. This proves the claim.

\medskip

Now we prove the theorem. The set of all dual frames of $F$ forms an affine space. For each
subset $I\subset\{1,\ldots,N\}$ with $|I|=n$, consider the determinant
\[
D_I(G)=\det(g_i)_{i\in I},
\]
where the columns are the vectors $g_i$, $i\in I$, written in a fixed basis of
$\mathcal H_n$. The condition $D_I(G)\neq 0$ is exactly the condition that $\{g_i\}_{i\in I}$ is a basis. By the claim, for each $I$ there exists at least one dual frame $G_{I}$ such that
$D_I(G_{I})\neq 0$. Hence $D_I$ is not the zero polynomial on the affine space of dual frames. There are only finitely many subsets $I$ with $|I|=n$. Therefore the product
\[
D_{I}=\prod_{|I|=n}D_I(G_{I})
\]
is also a nonzero polynomial. For each $I,$  let us define
\[Z_I = \{G \in D(F) : D_{I}(G) = 0 \}.\]
Since, $D_{I}$ is not the zero polynomial, $Z_{I}$ is a proper algebraic subset. Since there are only finitely many $I,$ by Lemma \ref{lem3point5}, $\bigcup\limits_{|I| =n}Z_{I} \neq D_{F}.$
Hence there exists a dual frame $G \in D(F) \setminus \cup_{|I| =n}Z_{I} $ such that  $D_I(G)\neq 0,\quad \text{for every } I\subset\{1,\ldots,N\},\ |I|=n.$ Thus every $n$ vectors of $G$ form a basis for $\mathcal H_n$. Therefore $G$ is $k-$erasure robust dual frame of $F$.
\end{proof}

The following example illustrates that a frame need not itself be maximally robust to erasures in order to admit a dual frame with maximal robustness. In particular, it demonstrates the conclusion of Theorem~\ref{robustthm3point5} in a simple finite-dimensional setting.
\begin{example}
Consider the frame
\[
F=\{f_1,f_2,f_3,f_4\}
=
\left\{
\begin{pmatrix}1\\0\end{pmatrix},
\begin{pmatrix}1\\0\end{pmatrix},
\begin{pmatrix}0\\1\end{pmatrix},
\begin{pmatrix}0\\1\end{pmatrix}
\right\}
\]
for $\mathbb R^2$. Here $N=4$, so Theorem \ref{robustthm3point5} guarantees the
existence of a dual frame which is $2$-erasure robust. Indeed, consider
\[
G=\{g_1,g_2,g_3,g_4\} = \left\{\begin{pmatrix}1\\1\end{pmatrix}, \begin{pmatrix}0\\-1\end{pmatrix} , \begin{pmatrix}1\\3\end{pmatrix},  \begin{pmatrix}-1\\-2\end{pmatrix}  \right\}.
\]

Since
\[
g_1+g_2=\begin{pmatrix}1\\0\end{pmatrix},
\qquad
g_3+g_4=\begin{pmatrix}0\\1\end{pmatrix},
\]
it follows immediately that $G$ is a dual frame of $F$.

Moreover, every pair of vectors in $G$ is linearly independent.
Hence $G$ is full spark and therefore remains a frame after the
erasure of any two elements. Thus $G$ is a $2$-erasure robust
dual of $F$.

Note that $F$ itself is not $2$-erasure robust, since, for
example, $\{f_1,f_2\}$ is linearly dependent. This illustrates
that even a frame which is not maximally robust may possess a
dual frame having the maximal robustness guaranteed by
Theorem~\ref{robustthm3point5}.
\end{example}

The full-spark property provides a natural characterization of maximal robustness of finite frames against erasures. The following result strengthens this connection by showing that $m-$erasure robustness can be characterized through a matrix of linear dependencies having a much stronger nonsingularity property. In particular, the relation matrix can be chosen to be totally nonsingular, so that every square submatrix is invertible.

\begin{theorem}
Let $F=\{f_j\}_{j=1}^{N}$ be a frame for an $n$-dimensional
Hilbert space $\mathcal H_n$, and let $1\leq m\leq N-n$.
Then the following statements are equivalent:
\begin{enumerate}
    \item $F$ is robust to $m$ erasures.

    \item There exists an $m\times N$ matrix $A=(a_{ij})$ such that $ \sum_{j=1}^{N} a_{ij}f_j=0,
    \quad i=1,\ldots,m,$
    and $A$ is totally non-singular i.e.  every square
    submatrix of $A$ is non-singular.
\end{enumerate}
\end{theorem}

\begin{proof}
$(2)\Rightarrow(1)$.
Since $A$ is totally nonsingular, in particular every
$m\times m$ column submatrix of $A$ is invertible.
Hence $A$ is full spark. By the characterization of
$m$-erasure robustness in terms of full-spark relation matrices,
it follows that $F$ is robust to $m$ erasures.

\medskip

$(1)\Rightarrow(2)$.
Assume that $F$ is robust to $m$ erasures.
Then there exists an
$m\times N$ matrix $A_0=(a_{ij}^{(0)})$
such that $\sum_{j=1}^{N}a_{ij}^{(0)}f_j=0,
\quad i=1,\ldots,m,$ and every $m\times m$ column submatrix of $A_0$ is invertible.

In particular, every collection of at most $m$ columns of $A_0$
is linearly independent. Indeed, if some $k<m$ columns were
linearly dependent, then after adjoining any $m-k$ additional
columns, the resulting $m$ columns would also be linearly
dependent, contradicting the full-spark property.

We now show that there exists an invertible matrix
$B\in GL(m,\mathbb F)$, where $\mathbb F=\mathbb R$ or
$\mathbb C$, such that
\[
A:=BA_0
\]
is totally nonsingular. Fix subsets $I\subseteq\{1,\ldots,m\},\quad
J\subseteq\{1,\ldots,N\},$ with $|I|=|J|=k,\quad 1\leq k\leq m.$ Consider the polynomial
\[
p_{I,J}(B)
=
\det\big((BA_0)_{I,J}\big),
\qquad B\in M_m(\mathbb F).
\]
We claim that $p_{I,J}$ is not the zero polynomial.

Since the columns of $(A_0)_J$ are linearly independent, we have $\operatorname{rank}(A_0)_J=k.$ Therefore there exist $k$ linear functionals on $\mathbb F^m$
whose restrictions to the column space of $(A_0)_J$ form a
basis of its dual space. These functionals may be chosen as the
rows of $B$ indexed by $I$ and then extended to a basis of
$(\mathbb F^m)^*$, yielding an invertible matrix $B$ for which
\[
\det\big((BA_0)_{I,J}\big)\neq0.
\]
Thus $p_{I,J}$ is a nonzero polynomial.

There are only finitely many pairs $(I,J)$. Hence the polynomial
\[
P(B)
=
\det(B)
\prod_{k=1}^{m}
\prod_{\substack{|I|=k\\ |J|=k}}
p_{I,J}(B)
\]
is a nonzero polynomial on $M_m(\mathbb F)$.
Since $\mathbb F=\mathbb R$ or $\mathbb C$ is infinite,
there exists $B$ such that $P(B)\neq 0.$ For this choice of $B$, we have $B\in GL(m,\mathbb F)$ and every square minor of $BA_0$ is nonzero. Set $A=BA_0.$ Since the rows of $A$ are linear combinations of the rows of
$A_0$, we still have $\sum_{j=1}^{N}a_{ij}f_j=0,
\quad i=1,\ldots,m.$
Moreover, $A$ is totally nonsingular. This proves $(2)$.
\end{proof}

Proposition~\ref{existanceprop3.3} guarantees the existence of a full-spark matrix
whose rows encode the linear dependencies among the frame vectors.
For certain applications, it is useful to strengthen this property
and construct a relation matrix for which every square minor is
nonzero. The following deterministic procedure uses a
Vandermonde-type change of basis in the relation space to obtain such a totally nonsingular matrix.

\medskip
\noindent
\textbf{Algorithm 3.4 (Construction of a totally nonsingular
$A$-balancing matrix).}
Let $F=\{f_j\}_{j=1}^N$ be an $m$-erasure robust frame for
$\mathcal H_n$.

\begin{enumerate}
    \item By Proposition~\ref{existanceprop3.3}, construct an
    $m\times N$ full-spark matrix
    \[
    B=(b_{\ell j})\in\mathbb F^{m\times N}
    \]
    whose rows are linear relations among the vectors of $F$. Thus
    \[
    BF=0,
    \]
    or, equivalently,
    $\sum_{j=1}^N b_{\ell j}f_j=0$ for
    $\ell=1,\ldots,m$.

    \item For each $j=1,\ldots,N$, associate with the $j$th column
    of $B$ the polynomial
    \[
    p_j(t)=b_{1j}+b_{2j}t+\cdots+b_{mj}t^{m-1}.
    \]
    Since $B$ is full spark, every collection of at most $m$
    columns of $B$ is linearly independent. Consequently, for every
    $J=\{j_1,\ldots,j_s\}\subseteq\{1,\ldots,N\}$ with
    $1\leq s\leq m$, the polynomials
    $p_{j_1},\ldots,p_{j_s}$ are linearly independent.

    \item Introduce parameters $t_1,\ldots,t_m$ and form the
    Vandermonde matrix
    \[
    V(t_1,\ldots,t_m)
    =
    \begin{pmatrix}
    1&t_1&t_1^2&\cdots&t_1^{m-1}\\
    1&t_2&t_2^2&\cdots&t_2^{m-1}\\
    \vdots&\vdots&\vdots&&\vdots\\
    1&t_m&t_m^2&\cdots&t_m^{m-1}
    \end{pmatrix}.
    \]
    Define
    \[
    A(t_1,\ldots,t_m)=V(t_1,\ldots,t_m)B.
    \]
    Its $(i,j)$th entry is $p_j(t_i)$.

    \item For $1\leq s\leq m$, let
    $I=\{i_1,\ldots,i_s\}\subseteq\{1,\ldots,m\}$ and
    $J=\{j_1,\ldots,j_s\}\subseteq\{1,\ldots,N\}$. The corresponding
    $s\times s$ minor of $A(t_1,\ldots,t_m)$ is
    \[
    q_{I,J}(t_1,\ldots,t_m)
    :=
    \det\big(p_{j_q}(t_{i_p})\big)_{p,q=1}^s.
    \]
    Since $p_{j_1},\ldots,p_{j_s}$ are linearly independent,
    $q_{I,J}$ is not the zero polynomial.

    \item Define the polynomial
    \[
    Q(t_1,\ldots,t_m)
    :=
    \prod_{s=1}^m
    \prod_{\substack{I\subseteq\{1,\ldots,m\},\,|I|=s\\
                     J\subseteq\{1,\ldots,N\},\,|J|=s}}
    q_{I,J}(t_1,\ldots,t_m).
    \]
    Each factor is a nonzero polynomial, and hence
    $Q$ is a nonzero polynomial. Therefore, over the infinite field
    $\mathbb F=\mathbb R$ or $\mathbb C$, there exists
    $(t_1,\ldots,t_m)\in\mathbb F^m$ such that
    \[
    Q(t_1,\ldots,t_m)\neq0.
    \]
    Choose any such tuple. Since the $m\times m$ minor corresponding
    to any $m$ columns is nonzero, the parameters $t_1,\ldots,t_m$
    are necessarily distinct; equivalently, one may require them to
    be distinct when making the choice.

    \item Set
    \[
    A=A(t_1,\ldots,t_m)
      =V(t_1,\ldots,t_m)B.
    \]
    Since $Q(t_1,\ldots,t_m)\neq0$, every square minor of $A$ is
    nonzero. Hence $A$ is totally nonsingular. Moreover,
    \[
    AF=V(t_1,\ldots,t_m)BF=0.
    \]
    Thus $F$ is $A$-balanced.
\end{enumerate}
We illustrate Algorithm~3.4 with a full-spark matrix which is
not totally nonsingular, and show explicitly how the algorithm
produces a totally nonsingular matrix.
\begin{example}
Consider the frame $F=\{f_1,f_2,f_3,f_4\}\subset \mathbb R^2,$
where
\[
f_1=
\begin{pmatrix}
-1\\-1
\end{pmatrix},
\qquad
f_2=
\begin{pmatrix}
-1\\-2
\end{pmatrix},
\qquad
f_3=
\begin{pmatrix}
1\\0
\end{pmatrix},
\qquad
f_4=
\begin{pmatrix}
0\\1
\end{pmatrix}.
\]
Every pair of vectors in $F$ is linearly independent. Hence $F$
is full spark and it is $2$-erasure robust.
Consider the matrix
\[
B=
\begin{pmatrix}
1&0&1&1\\
0&1&1&2
\end{pmatrix}.
\]
$B$ is not totally nonsingular. The two rows of $B$ give the relations
\[
f_1+f_3+f_4=0, \quad \text{and} \quad f_2+f_3+2f_4=0.
\]

Consequently, $BF=0.$ Following Algorithm~3.4, the polynomials associated with the
columns of $B$ are
\[
p_1(t)=1,\qquad
p_2(t)=t,\qquad
p_3(t)=1+t,\qquad
p_4(t)=1+2t.
\]
Choose $t_1=1,\quad t_2=3.$ Then $V=
\begin{pmatrix}
1&1\\
1&3
\end{pmatrix},$
and hence
\[
A=VB
=
\begin{pmatrix}
1&1\\
1&3
\end{pmatrix}
\begin{pmatrix}
1&0&1&1\\
0&1&1&2
\end{pmatrix}
=
\begin{pmatrix}
1&1&2&3\\
1&3&4&7
\end{pmatrix}.
\]

It is easy to verify that $A$ is totally nonsingular.  Finally,
\[
AF=VBF=0.
\]
Therefore $F$ is $A$-balanced, where $A=
\begin{pmatrix}
1&1&2&3\\
1&3&4&7
\end{pmatrix}$
is totally nonsingular, even though the original balancing matrix $B=
\begin{pmatrix}
1&0&1&1\\
0&1&1&2
\end{pmatrix}$
is not totally nonsingular.
\end{example}

\begin{example}
Consider the frame
\[
F=\left\{
f_1=\begin{pmatrix}1\\0\end{pmatrix},
f_2=\begin{pmatrix}0\\1\end{pmatrix},
f_3=\begin{pmatrix}1\\1\end{pmatrix}
\right\}\subset\mathbb R^2
\]
and let
\[
A=\begin{pmatrix}1&1&1\end{pmatrix}.
\]
Then
\[
A[f_1,f_2,f_3]^T=f_1+f_2+f_3
=\begin{pmatrix}2\\2\end{pmatrix}\neq0,
\]
so $F$ does not have vanishing $A$-moment.

The canonical dual of $F$ is
\[
\left\{
\frac13\begin{pmatrix}2\\-1\end{pmatrix},
\frac13\begin{pmatrix}-1\\2\end{pmatrix},
\frac13\begin{pmatrix}1\\1\end{pmatrix}
\right\},
\]
whose elements do not sum to zero. Hence the canonical dual does
not have vanishing $A$-moment either.

On the other hand, every dual frame of $F$ is of the form
\[
G_h=
\left\{
\begin{pmatrix}1\\0\end{pmatrix}-h,\,
\begin{pmatrix}0\\1\end{pmatrix}-h,\,
h
\right\},
\qquad h\in\mathbb R^2.
\]
Taking $h=\begin{pmatrix}1\\1\end{pmatrix},$
we obtain the dual frame
\[
G=
\left\{
\begin{pmatrix}0\\-1\end{pmatrix},
\begin{pmatrix}-1\\0\end{pmatrix},
\begin{pmatrix}1\\1\end{pmatrix}
\right\}.
\]
Clearly, $g_1+g_2+g_3=0,$ and hence $A[g_1,g_2,g_3]^T=0.$

Thus $F$ admits a dual frame with vanishing $A$-moment even though
neither $F$ nor its canonical dual has vanishing $A$-moment.
\end{example}

\section{ Frame vanishing moments and connection with graphs}
An important special case of vanishing $A$-moments arises when the matrix $A$ is chosen to be a Vandermonde-type matrix. This leads to the notion of $k$-vanishing moments, which plays a central role in wavelet theory, approximation theory, and signal processing. In the finite-dimensional frame setting, $k$-vanishing moments impose a collection of polynomial cancellation conditions on the frame vectors and provide a natural extension of the concept of balanced frames.

Let $F=\{f_i\}_{i=1}^{N}$ be a frame for an $n$-dimensional Hilbert space $\mathcal H_n$. For each integer $r\geq 0$, define $v_r=(1^r,2^r,\ldots,N^r)^t\in\mathbb C^N.$ We say that $F$ has \emph{$k$-vanishing moments} if $\sum_{i=1}^{N} i^r f_i =0, \quad r=0,1,\ldots,k-1.$ Equivalently,
$\Theta_F^{*}v_r=0, \quad r=0,1,\ldots,k-1.$
Observe that the case $k=1$ reduces to the balanced frame condition $\sum_{i=1}^{N}f_i=0.$ For $0\le r\le k-1$, let $u_r=\sum_{i=1}^{N}i^r f_i,
\quad v_r=(1^r,2^r,\ldots,N^r)^t,$ and define $P_k=\operatorname{span}\{v_0,v_1,\ldots,v_{k-1}\}.$
Since $v_0,v_1,\ldots,v_{k-1}$ are columns of a Vandermonde matrix, they are linearly independent whenever $N\ge k$, and hence $\dim P_k=k.$\\
\par The following proposition characterizes the existence of dual frames possessing $k$-vanishing moments in terms of the geometry of the analysis operator of the original frame.

\begin{proposition}
Let $F=\{f_i\}_{i=1}^N$ be a frame for $\mathcal H_n$ and let
$1\leq k\leq N$. Then $F$ admits a dual frame with
$k$-vanishing moments if and only if $P_k\cap\operatorname{Range}(\Theta_F)=\{0\}.$

In this case, necessarily $k\leq N-n$.
\end{proposition}
\begin{proof}
Suppose first that $G=\{g_i\}_{i=1}^N$ is a dual frame of $F$ having 
$k$-vanishing moments. That is, $\sum_{i=1}^N i^r g_i=0,
\quad r=0,1,\dots,k-1.$ Equivalently, $\Theta_{G}^* v_r=0, \quad r=0,1,\dots,k-1.$ Hence $P_k\subseteq \ker \Theta_{G}^* = (\operatorname{Range}(\Theta_G))^\perp .$  Since $G$ is a dual frame of $F$, we have $\Theta_{G}^* \Theta_F=I_{\mathcal H_n}.$
It follows that $\Theta_{G}^*$ is injective on $\operatorname{Range}(\Theta_F)$. Indeed, if
$y\in \operatorname{Range}(\Theta_F)$ and $\Theta_{G}^*y=0$, then there exists 
$x\in\mathcal H_n$ such that $y=\Theta_Fx$. Thus $0=\Theta_{G}^*y=\Theta_{G}^* \Theta_Fx=x,$
and hence $y=\Theta_Fx=0$. Therefore,
\[
\ker \Theta_{G}^*\cap \operatorname{Range}(\Theta_F)=\{0\}.
\]
Since $P_k\subseteq \ker \Theta_{G}^*$, we obtain $P_k\cap \operatorname{Range}(\Theta_F)=\{0\}.$ Moreover, since $\Theta_{G}^* \Theta_F=I_{\mathcal H_n}$, we have $\operatorname{rank}(\Theta_{G})=n.$ Thus
$\dim \ker \Theta_{G}^*=N-n.$ Because $P_k\subseteq \ker \Theta_{G}^*$, it follows that
\[
\dim P_k\leq \dim \ker \Theta_{G}^*=N-n.
\]
Since the vectors $v_0,v_1,\dots,v_{k-1}$ are linearly independent, 
$\dim P_k=k.$ Therefore, $k\leq N-n.$

Conversely, assume that $P_k\cap \operatorname{Range}(\Theta_F)=\{0\}.$ Since $\dim P_k=k$ and $\dim \operatorname{Range}(\Theta_F)=n$, we can choose a 
subspace $Y\subseteq \mathbb C^N$ such that $\mathbb C^N=\operatorname{Range}(\Theta_F)\oplus Y$ and $P_k\subseteq Y.$
Define a linear operator $T:\mathbb C^N\to \mathcal H_n$ by
\[
T(\Theta_Ff)=f,
\quad f\in\mathcal H_n, \qquad \text{and} \; T|_Y=0.
\]

Then $T \Theta_F=I_{\mathcal H_n}.$ Hence $T$ is the synthesis operator of some dual frame  $G=\{g_i\}_{i=1}^N$ of $F$. Since $P_k\subseteq Y\subseteq \ker T$, we have $Tv_r=0, \quad r=0,1,\dots,k-1.$ Equivalently, $\sum_{i=1}^N i^r g_i=0,
\quad r=0,1,\dots,k-1.$ Therefore, $G$ has $k$-vanishing moments. This completes the proof.
\end{proof}

\begin{remark}
    $F$ has $k$-vanishing moments iff $P_k\subseteq \ker \Theta_F^*$.
\end{remark}

The notion of $k$-vanishing moments can be viewed from different perspectives. Besides the moment conditions themselves, they admit equivalent descriptions through polynomial cancellation properties and factorization of the analysis operator. The following propositions provide these alternative characterizations.
\begin{proposition}
Let $F=\{f_i\}_{i=1}^N$ be a frame for $\mathcal H_n$. Then $F$ has $k$-vanishing moments if and only if For every polynomial $p$ with $\deg(p)\leq k-1$,\; $\sum_{i=1}^N p(i)f_i=0.$
\end{proposition}

\begin{proof}
Suppose first that $F$ has $k$-vanishing moments. Then, by definition, $\sum_{i=1}^N i^r f_i=0, \quad r=0,1,\ldots,k-1.$ Let $p$ be any polynomial with $\deg(p)\leq k-1$. Then $p(t)=\sum_{r=0}^{k-1} a_r t^r, $ for some scalars $a_0,a_1,\ldots,a_{k-1}$. Therefore,
\[
\sum_{i=1}^N p(i)f_i
=
\sum_{i=1}^N \left(\sum_{r=0}^{k-1} a_r i^r\right)f_i
=
\sum_{r=0}^{k-1} a_r \left(\sum_{i=1}^N i^r f_i\right)
=0.
\]

Conversely, suppose that for every polynomial $p$ with $\deg(p)\leq k-1,\; \sum_{i=1}^N p(i)f_i=0.$ Taking $p(t)=t^r$ for $r=0,1,\ldots,k-1$, we obtain $\sum_{i=1}^N i^r f_i=0,
\qquad r=0,1,\ldots,k-1.$ Hence $F$ has $k$-vanishing moments.
\end{proof}

\begin{proposition}
Let $F=\{f_i\}_{i=1}^N$ be a frame for  $\mathcal H_n$ and $V_k=[v_0,v_1,\ldots,v_{k-1}]\in \mathbb F^{N\times k}.$ Then the following statements are equivalent:
\begin{enumerate}
    \item $F$ has $k$-vanishing moments.
    \item There exists a matrix $B\in \mathbb F^{N\times (N-k)}$ with
    $\operatorname{rank}(B)=N-k$ such that $V_k^*B=0$   and there exists a frame $H=\{h_j\}_{j=1}^{N-k}$ for $\mathcal H_d$ such that $\Theta_F^*=T_H^*B^*.$
    Equivalently,
    \[
    f_i=\sum_{j=1}^{N-k}\overline{b_{ij}}\,h_j,
    \qquad i=1,2,\ldots,N.
    \]
\end{enumerate}
\end{proposition}

\begin{proof}
Assume that $F$ has $k$-vanishing moments. Then $\Theta_F^*v_r=0,
\qquad r=0,1,\ldots,k-1.$ Hence $\operatorname{Ran}(V_k)\subseteq \ker \Theta_F^*. $ Choose a matrix $B\in \mathbb F^{N\times (N-k)}$ whose columns form a basis for $\operatorname{Ran}(V_k)^\perp$. Then $V_k^*B=0$ and $B$ has full column rank $N-k$. Since $\mathbb F^N=\operatorname{Ran}(V_k)\oplus \operatorname{Ran}(B),$ and  $\operatorname{Ran}(V_k)\subseteq \ker \Theta_F^*$, the operator $\Theta_F^*$ is completely determined by its values on $\operatorname{Ran}(B)$. Define
\[
T_H^*:=\Theta_F^*(B^*)^\dagger,
\]
where $(B^*)^\dagger$ denotes a right inverse of $B^*$. Then $T_H^*B^*=\Theta_F^*.$\\
Let $H=\{h_j\}_{j=1}^{N-k}$ be the sequence determined by the synthesis operator $T_H^*$,
that is, $T_H^*e_j=h_j,\quad j=1,2,\ldots,N-k.$ Since $F$ is a frame for $\mathcal H_d$, we have $\operatorname{rank}(\Theta_F^*)=d.$
Also $B$ has full column rank and $\Theta_F^*=T_H^*B^*$, so $T_H^*$ has rank $d$. Therefore $H$ spans $\mathcal H_d$, and hence $H$ is a frame for $\mathcal H_d$. Finally, writing $B=(b_{ij})$, the identity $\Theta_F^*=T_H^*B^*$ gives
\[
f_i=\Theta_F^*e_i
=
T_H^*B^*e_i
=
T_H^*
\begin{pmatrix}
\overline{b_{i1}}\\
\overline{b_{i2}}\\
\vdots\\
\overline{b_{i,N-k}}
\end{pmatrix}
=
\sum_{j=1}^{N-k}\overline{b_{ij}}\,h_j.
\].

Conversely, for $r=0,1,\ldots,k-1$, we have $\Theta_F^*v_r=T_H^*B^*v_r.$ Since $V_k^*B=0$, it follows that $B^*V_k=0.$ Hence $B^*v_r=0,\qquad r=0,1,\ldots,k-1.$ Therefore, $\Theta_F^*v_r=T_H^*B^*v_r=0,
\quad r=0,1,\ldots,k-1.$ Equivalently, $\sum_{i=1}^N i^r f_i=0,
\quad r=0,1,\ldots,k-1.$ Thus $F$ has $k$-vanishing moments.

\end{proof}

The previous results characterize $k$-vanishing moments in terms of polynomial cancellations and operator factorizations. We now present a characterization for the existence of dual frames with $k$-vanishing moments using the ranges of suitable analysis operators associated with the moment vectors. Recall that for $0\leq j\leq k-1$, define $u_j=\sum_{i=1}^N i^j f_i$ and $v_j=(1^j,2^j,\dots,N^j)^t.$ Let  $U=\{u_j\}_{j=0}^{k-1},$ $V=\{v_j\}_{j=0}^{k-1}.$

\begin{proposition}
Let $F=\{f_i\}_{i=1}^N$ be a frame for an $n$-dimensional Hilbert space
$\mathcal H_n$.  If $F$ has a dual frame with $k$-vanishing moments then
$\operatorname{Range}(\Theta_U)\subseteq \operatorname{Range}(\Theta_V).$
\end{proposition}

\begin{proof}
Let $G=\{g_i\}_{i=1}^N$ be a dual frame of $F$. Then $T_G^*\Theta_F=I_{\mathcal H_n}.$ The condition that $G$ has $k$-vanishing moments means that $\sum_{i=1}^N i^j g_i=0,
\quad j=0,1,\ldots,k-1.$ Equivalently, $T_G^*v_j=0, \quad j=0,1,\ldots,k-1,$ where
$v_j=(1^j,2^j,\ldots,N^j)^t\in\mathbb C^N.$ Now, by the definition of $u_j$, for every $f\in\mathcal H_n$ we have
\[
\langle f,u_j\rangle
=
\left\langle f,\sum_{i=1}^N i^j f_i\right\rangle
=
\sum_{i=1}^N i^j\langle f,f_i\rangle .
\]
Hence $\Theta_U f
=
\left(
\sum_{i=1}^N i^0\langle f,f_i\rangle,
\sum_{i=1}^N i^1\langle f,f_i\rangle,
\ldots,
\sum_{i=1}^N i^{k-1}\langle f,f_i\rangle
\right).$ Suppose first that $F$ has a dual frame $G$ with $k$-vanishing moments. Then $T_G^*v_j=0,\quad j=0,1,\ldots,k-1.$ Let $\alpha\in \operatorname{Range}(\Theta_U)$. Then there exists $f\in\mathcal H_n$ such that $\alpha=\Theta_U f.$ Thus
$\alpha = \big(\langle f,u_0\rangle,\langle f,u_1\rangle,\ldots, \langle f,u_{k-1}\rangle\big).$ Since $u_j=\sum_{i=1}^N i^j f_i,$ we have $\langle f,u_j\rangle
= \left\langle f,\sum_{i=1}^N i^j f_i\right\rangle = \sum_{i=1}^N i^j \langle f,f_i\rangle .$ Let us take
\[
c=(\langle f,f_1\rangle,\langle f,f_2\rangle,\ldots,
\langle f,f_N\rangle)^t
=\Theta_F f\in\mathbb C^N.
\]
Then, for each $j=0,1,\ldots,k-1$, \,$\alpha_j = \sum_{i=1}^N i^j \langle f,f_i\rangle = \langle c,v_j\rangle .$ Therefore, $\alpha = \big(\langle c,v_0\rangle,\langle c,v_1\rangle,\ldots, \langle c,v_{k-1}\rangle\big) = \Theta_V c.$ Hence $\alpha\in \operatorname{Range}(\Theta_V).$ Since $\alpha\in \operatorname{Range}(\Theta_U)$ was arbitrary, we conclude that $\operatorname{Range}(\Theta_U)\subseteq \operatorname{Range}(\Theta_V).$

\end{proof}

\begin{example}
Consider the frame \(
F=\left\{
\begin{pmatrix}2\\1\end{pmatrix},
\begin{pmatrix}1\\2\end{pmatrix},
\begin{pmatrix}1\\3\end{pmatrix},
\begin{pmatrix}1\\5\end{pmatrix}
\right\}
\)
 and a dual \(G=\left\{
\begin{pmatrix}1\\0\end{pmatrix},
\begin{pmatrix}-2\\1\end{pmatrix},
\begin{pmatrix}1\\-2\end{pmatrix},
\begin{pmatrix}0\\1\end{pmatrix}
\right\}\) for $\mathbb R^2$.
It is easy to see that
\[
\sum_{i=1}^{4}g_i=0
\qquad\text{and}\qquad
\sum_{i=1}^{4}i\,g_i=0.
\]
Thus $G$ has two vanishing moments. Equivalently, if
\[
P_2=\operatorname{span}
\left\{
\begin{pmatrix}1\\1\\1\\1\end{pmatrix},
\begin{pmatrix}1\\2\\3\\4\end{pmatrix}
\right\},
\]
then $P_2\subseteq\ker\Theta_G^*.$ Consequently, for every polynomial $p$ of degree at most one, $\sum_{i=1}^{4}p(i)g_i=0.$ On the other hand, $\sum_{i=1}^{4}f_i\neq0,$ so $F$ itself does not even have one vanishing moment. Hence this
example illustrates that a frame need not possess vanishing moments
itself in order to admit a dual frame with higher-order vanishing
moments.
\end{example}

\subsection{Graph Frame}

We begin by recalling the connection between finite frames and
finite simple graphs. Graph-theoretic structures provide a natural
way of constructing and analyzing finite frames through matrices
associated with graphs, particularly the graph Laplacian. This
connection allows geometric and linear-algebraic properties of a
frame to be studied in terms of combinatorial properties of the
underlying graph.

Let $\Gamma$ be a finite simple graph with vertex set $V(\Gamma)=\{v_1,\ldots,v_N\}.$ The adjacency matrix of $\Gamma$ is the matrix $A(\Gamma)=(a_{ij})_{i,j=1}^N$ defined by
\[
a_{ij}=
\begin{cases}
1, & \text{if } v_i \text{ and } v_j \text{ are adjacent},\\
0, & \text{otherwise}.
\end{cases}
\]
Let $D(\Gamma)=\operatorname{diag}(d_1,\ldots,d_N),$
where $d_i$ denotes the degree of the vertex $v_i$. The
\emph{Laplacian matrix} of $\Gamma$ is defined by $L(\Gamma)=D(\Gamma)-A(\Gamma).$ It is well known that $L(\Gamma)$ is positive semidefinite and
that the multiplicity of the eigenvalue $0$ is equal to the number
of connected components of $\Gamma$. Consequently, if $\Gamma$
has $c$ connected components, then $\operatorname{rank} L(\Gamma)=N-c.$
On the frame-theoretic side, let
$F=\{f_i\}_{i=1}^N$ be a frame for a finite-dimensional Hilbert
space $\mathcal H$. Its Gramian is the $N\times N$ matrix $G_F=\big(\langle f_j,f_i\rangle\big)_{i,j=1}^N
    =\Theta_F\Theta_F^*.$
We say that $F$ is a \emph{graph frame} if there exists a finite
simple graph $\Gamma$ such that $G_F=L(\Gamma).$ In other words, a frame is a graph frame precisely when its
Gramian coincides with the Laplacian matrix of some finite simple
graph. In this case, we say that $F$ is generated by, or associated
with, the graph $\Gamma$. For a more detailed discussion of graph frames and their properties, we refer the reader to \cite{shankha}.

\begin{example}
Consider the balanced frame
\[
F=
\left\{
\begin{pmatrix}1\\0\end{pmatrix},
\begin{pmatrix}0\\1\end{pmatrix},
\begin{pmatrix}-1\\1\end{pmatrix},
\begin{pmatrix}0\\-2\end{pmatrix}
\right\}
\subset \mathbb R^2.
\]

Its Gram matrix is
\[
G_F=
\begin{pmatrix}
1&0&-1&0\\
0&1&1&-2\\
-1&1&2&-2\\
0&-2&-2&4
\end{pmatrix}.
\]

Since $(G_F)_{23}=1>0$, the Gram matrix has a positive off-diagonal entry. However,
every graph Laplacian has non-positive off-diagonal entries. Therefore $G_F$ cannot be the Laplacian matrix of any graph.
Hence $F$ is a balanced frame which is not a graph frame.
\end{example}

It is clear from the preceding discussion that not every frame can
be realized as a graph frame, since the Gramian of a graph frame
must possess the specific structural properties of a graph
Laplacian. This naturally raises the question of which frames admit
such a graph-theoretic realization. The following propositions and
theorems provide conditions under which a frame can be identified
with, or constructed from, the Laplacian of a finite graph, thereby
clarifying the relationship between balanced frames and graph
frames.

The notion of linear connectedness plays an important role in the structure of balanced frames and graph frames. It provides a natural graph-theoretic interpretation of a frame by connecting two frame vectors whenever their inner product is nonzero. The following definition and result, taken from \cite{sali}, will be useful in our study of balanced frames and graph frames.

A frame $F=\{f_i\}_{i=1}^N$ is said to be \emph{linearly connected} if, for every nonempty proper subset $\Lambda\subsetneq\{1,2,\ldots,N\},$ we have $\operatorname{span}\{f_i:i\in\Lambda\} \cap
\operatorname{span}\{f_i:i\in\Lambda^c\} \neq
\{0\}.$ Equivalently, $F$ is linearly connected if it is impossible to partition the frame into two nonempty subsets whose spans intersect trivially.\\

Let $F=\{f_i\}_{i=1}^{N}$ be a finite sequence of nonzero vectors in a Hilbert space $\mathcal H_n$. Two vectors $f_i$ and $f_j$ in $F$ are said to be \emph{linearly $F$-connected} (or simply
\emph{connected})\cite{sali} if there exist vectors $f_{k_1},\ldots,f_{k_\ell}\in F$ such that $\{f_j,f_{k_1},\ldots,f_{k_\ell}\}$ is linearly independent and $f_i = c\,f_j+\sum_{m=1}^{\ell} c_m f_{k_m},$ where all coefficients $c,c_1,\ldots,c_\ell$ are nonzero. We write $f_i \overset{F}{\longleftrightarrow} f_j$ whenever $f_i$ and $f_j$ are $F$-connected.
The sequence $F$ is called \emph{linearly connected} if every two vectors in $F$ are $F$-connected.

\begin{lemma}\label{lem3point6}
Let $F=\{f_i\}_{i=1}^N$ be a linearly connected graph frame for $\mathcal H_n$. Then the associated graph $\Gamma(F)$ is connected.
\end{lemma}

\begin{proof}
Suppose, to the contrary, that the associated graph $\Gamma(F)$ is disconnected. Then there exists a nonempty proper subset $\Lambda\subsetneq\{1,2,\ldots,N\}$ such that there is no edge joining a vertex in $\Lambda$ with a vertex in $\Lambda^c$. Hence, $\langle f_i,f_j\rangle=0,
\quad
\forall\, i\in\Lambda,\ j\in\Lambda^c.$ Define the subspaces
\[
\mathcal H_{\Lambda} = \operatorname{span}\{f_i:i\in\Lambda\} \qquad \text{and} \qquad
\mathcal H_{\Lambda^c} = \operatorname{span}\{f_i:i\in\Lambda^c\}.
\]

It is easy to see that $\mathcal H_{\Lambda}\perp\mathcal H_{\Lambda^c},$ which implies $\mathcal H_{\Lambda}\cap\mathcal H_{\Lambda^c}
=\{0\}.$ This contradicts the assumption that $F$ is linearly connected. Therefore, the associated graph $\Gamma(F)$ must be connected.
\end{proof}

\begin{proposition}
Let $F=\{f_i\}_{i=1}^{N}$ be a balanced frame for $\mathcal H_n$. Suppose that $F$ is linearly connected and
$\langle f_i,f_j\rangle\in\{-1,0\},$ for all
$ i\neq j.$ Then $F$ is a graph frame and $N=n+1.$
\end{proposition}

\begin{proof}
Let $\Gamma$ be the associated graph of $F$. Since $F$ is linearly connected, Lemma~\ref{lem3point6} implies that $\Gamma$ is connected. Let $L(\Gamma)$ denote the Laplacian matrix of $\Gamma$, and let $\mathbf{1}=(1,1,\ldots,1)^t.$
Since $\Gamma$ is connected, it is well known that
\[
\ker L(\Gamma)=\operatorname{span}\{\mathbf{1}\}.
\]

Indeed, let $x=(x_1,\ldots,x_N)^t\in\ker L(\Gamma).$ Then $L(\Gamma)x=0.$
Since $L(\Gamma)=D-A,$ where $D$ is the degree matrix and $A=(a_{ij})$ is the adjacency matrix of $\Gamma$, we obtain
\[
x^tL(\Gamma)x
=
x^tDx-x^tAx
=0.
\]

Now,
\[
x^tDx
=
\sum_{i=1}^{N}d_ix_i^2
=
\sum_{i=1}^{N}
\left(\sum_{j=1}^{N}a_{ij}\right)x_i^2
=
\sum_{i,j=1}^{N}a_{ij}x_i^2,
\]
while $x^tAx
=
\sum_{i,j=1}^{N}a_{ij}x_ix_j.$
Hence, $\sum_{i,j=1}^{N}a_{ij}(x_i^2-x_ix_j)=0.$

Since $A$ is symmetric, we have
\[
\begin{aligned}
2x^tL(\Gamma)x
=
\sum_{i,j=1}^{N}a_{ij}(x_i^2-x_ix_j)
+\sum_{i,j=1}^{N}a_{ij}(x_j^2-x_ix_j)=
\sum_{i,j=1}^{N}
a_{ij}(x_i-x_j)^2.
\end{aligned}
\]

Since $a_{ij}\ge0$, it follows that
\[
a_{ij}(x_i-x_j)^2=0,\quad \text{for every $i,j$. }
\]
Therefore, $x_i=x_j,$ whenever $a_{ij}=1$, that is, whenever the vertices $i$ and $j$ are adjacent. Since $\Gamma$ is connected, every pair of vertices is joined by a path. Hence $x_i=x_j,$ for all $1\le i,j\le N$, so that $x\in\operatorname{span}\{\mathbf 1\}.$ Therefore,
\[
\ker G(F)=\ker L(\Gamma) = \operatorname{span}\{\mathbf1\}.
\]

Hence $\dim\ker G(F)=1.$ Since $\operatorname{rank}(G(F))
=
\dim\mathcal H_n
=
n,$ the rank-nullity theorem gives
\[
N
=
\operatorname{rank}(G(F))
+
\dim\ker G(F)
=
n+1.
\]

\end{proof}

Let $F=\{f_i\}_{i=1}^N$ be a frame for an $n$-dimensional Hilbert space
$\mathcal H_n$. Suppose that $F$ has $m$ linearly connected components
$F_1,F_2,\ldots,F_m,$ where $F_j=\{f_i:i\in \Lambda_j\},\quad 
\{\Lambda_j\}_{j=1}^m$ is a partition of $\{1,2,\ldots,N\}$. Let $W_j=\operatorname{span}F_j,\quad d_j=\dim W_j,\quad N_j=|\Lambda_j|.$

\begin{theorem}\cite{sali}\label{cor3point7}
Let $F=\{f_i\}_{i=1}^{N}$ be a sequence in a Hilbert space $\mathcal H_n$. Then there exists a partition $\{\Lambda_j\}_{j=1}^{J}$ of $\{1,2,\ldots,N\}$, unique up to permutations, such that

\begin{enumerate}
    \item Each subsequence $F_j:=\{f_i\}_{i\in\Lambda_j}$  is linearly connected;

    \item If $ H_j=\operatorname{span}\{f_i:\, i\in\Lambda_j\},$ then
$\mathcal{H}_n = H_1\oplus H_2\oplus\cdots\oplus H_J .$
\end{enumerate}
\end{theorem}

The linearly connected decomposition provides a natural way to characterize when a balanced frame can be realized as a graph frame. The following theorem shows that this occurs precisely when each linearly connected component has excess one.

\begin{theorem}\label{graphframethm4.10}
Let $F=\{f_i\}_{i=1}^{N}$ be a balanced frame for $\mathcal H_n$. Suppose that
\[
F=F_1\cup F_2\cup\cdots\cup F_m,
\]
be the linearly connected decomposition of $F$. Let $\mathcal H_j=\operatorname{span}(F_j),\quad
N_j=|\Lambda_j|,\quad
d_j=\dim\mathcal H_j.$
Then $F$ is equivalent to a graph frame if and only if $N_j=d_j+1,\quad j=1,\ldots,m.$

\end{theorem}

\begin{proof}
Suppose first that $F$ is equivalent to a graph frame. Then there exist an invertible operator $T$ such that $G(TF) = L(\Gamma),$ for some graph $\Gamma.$ Therefore, each $TF_j$ corresponds to a connected component $\Gamma_j$ of $\Gamma$ i.e. $G(F_j)=L(\Gamma_j).$ Since $\Gamma_j$ is connected,
$\dim\ker L(\Gamma_j)=1.$ On the other hand, $\operatorname{rank}(G(F_j))
=
\dim\mathcal H_j
=
d_j.$ Hence, $\dim\ker G(F_j)
=
N_j-d_j.$ Since $\ker G(F_j)=\ker L(\Gamma_j),$ we obtain
\[
N_j-d_j=1, \quad
j=1,\ldots,m.
\]

Conversely, suppose that $N_j=d_j+1,
\quad
j=1,\ldots,m.$

Let $\Theta_j^*:\mathbb{R}^{N_j}\longrightarrow\mathcal H_j$
be the synthesis operator of the frame $F_j$. Since $\operatorname{rank}(\Theta_j^*)
=
d_j
=
N_j-1,$
we have $\dim\ker\Theta_j^*=1.$ Moreover, since $F_j$ is balanced, $\sum_{i\in\Lambda_j}f_i=0,$ and therefore $(1,\ldots,1)^t\in\ker\Theta_j^*.$
Since the kernel is one-dimensional, $\ker\Theta_j^* = \operatorname{span}\{(1,\ldots,1)^t\}.$

Now choose a connected graph $\Gamma_j$ with $N_j$ vertices, and let
\[
G_j=\{g_i^{(j)}\}_{i\in\Lambda_j}
\]
be the corresponding graph frame. Since $\Gamma_j$ is connected,
\[
\ker\Theta_{G_j}^*
=
\ker L(\Gamma_j)
=
\operatorname{span}\{(1,\ldots,1)^t\}.
\]

Hence,
\[
\ker\Theta_j^*
=
\ker\Theta_{G_j}^*.
\]

By the characterization of finite frame equivalence via synthesis operators, two finite frames spanning the same space are equivalent if and only if their synthesis operators have the same kernel. Therefore, there exists an invertible operator
\[
T_j:\mathcal H_j\longrightarrow\mathcal K_j
\]
such that
\[
T_jf_i=g_i^{(j)},
\qquad
i\in\Lambda_j.
\]

Since $\mathcal H_n
=
\mathcal H_1\oplus\cdots\oplus\mathcal H_m,$
define
\[
T
=
T_1\oplus\cdots\oplus T_m.
\]
Then $T$ is invertible and maps $F$ onto the graph frame associated with the disjoint union
$\Gamma
=
\Gamma_1\sqcup\cdots\sqcup\Gamma_m.$ Hence, $F$ is equivalent to a graph frame.
\end{proof}

\noindent The following example illustrates the componentwise characterization given in Theorem~\ref{graphframethm4.10}. In particular, it shows how the condition $N_j=d_j+1$ on each linearly connected component leads to a graph-frame representation of the entire balanced frame.

\begin{example}
Let $\mathcal H_4=\mathbb R^2\oplus\mathbb R^2$, and consider
the balanced frame
\[
F=F_1\cup F_2,
\]
where
\[
F_1=
\left\{
(1,0,0,0),\,
(0,1,0,0),\,
(-1,-1,0,0)
\right\}
\]
and
\[
F_2=
\left\{
(0,0,1,0),\,
(0,0,1,1),\,
(0,0,-2,-1)
\right\}.
\]
The families $F_1$ and $F_2$ are the linearly connected
components of $F$. Moreover,
\[
H_j=\operatorname{span}(F_j),\qquad
d_j=\dim H_j=2,\qquad N_j=3,
\quad j=1,2.
\]
Thus
\[
N_j=d_j+1,\qquad j=1,2.
\]
It follows from Theorem~\ref{graphframethm4.10} that $F$ is equivalent to a graph
frame.

More specifically, each component $F_j$ is equivalent to a
graph frame associated with a connected graph on three vertices.
Consequently, $F$ is equivalent to a graph frame associated with
the disjoint union of two connected three-vertex graphs.

This example also illustrates the componentwise nature of
Theorem~\ref{graphframethm4.10}: graph-frame equivalence is determined not merely by
the total redundancy of $F$, but by the precise redundancy
$N_j-d_j=1$ within each linearly connected component.
\end{example}

\section{Conclusion and Future Directions}

In this paper, we introduced the notion of vanishing $A$-moments
and developed a framework connecting moment conditions, dual frames,
dilation theory, and robustness against erasures. In particular, we
obtained equivalent characterizations for the existence of a
vanishing $A$-moment dilation in terms of dual frames and the
relative position of the row space of $A$ and the range of the
analysis operator. We also investigated the role of linear
dependencies in erasure robustness and showed how full-spark and
totally nonsingular matrices naturally arise in this setting.

As an important specialization, we considered higher-order
vanishing moments and characterized the existence of dual frames
satisfying such moment conditions. The case of the zeroth moment
leads naturally to balanced frames and provides a connection with
frames arising from graph Laplacians. Using the linearly connected
decomposition of a balanced frame, we obtained a componentwise
characterization of those balanced frames that are equivalent to
graph frames.

Several directions arising from the present work appear worthy of
further investigation. One natural problem is to study vanishing
$A$-moment dilations in infinite-dimensional Hilbert spaces and to
determine which of the finite-dimensional characterizations persist
in that setting. It would also be interesting to investigate
structured choices of the matrix $A$, particularly those arising
from harmonic analysis, sampling theory, wavelets, and structured
linear dependencies, and to understand how the algebraic properties
of $A$ influence the geometry and robustness of the associated
frames.

Another direction concerns the construction and optimization of
dual frames satisfying prescribed higher-order vanishing moments.
In particular, one may seek duals which simultaneously satisfy
moment constraints and optimize robustness or reconstruction error
under erasures. Finally, the connection between balanced frames and
graph Laplacians suggests studying more general graph-theoretic
models, including weighted graphs and other graph-associated
operators. Such extensions may lead to broader classes of
structured frames and provide further connections between frame
theory, harmonic analysis, and spectral graph theory.

\section{Acknowledgment}
The first author acknowledges partial support from a Simons travel award and the second  author is grateful to the Mohapatra Family Foundation and the College of Graduate Studies at the University of Central Florida for supporting his postdoctoral fellowship during the
course of this research.

\section{Declaration} 
\noindent Ethical approval: Not applicable.\\
Data Availability Statement: Not Applicable.\\
Conflict of Interest: Not Applicable.\\
Clinical Trial: Not applicable.

\section{Funding}

This research did not receive any specific grant from funding agencies in the public, commercial, or not-for-profit sectors.

\end{document}